\documentclass[11pt,reqno]{amsart}
\usepackage{latexsym,amssymb,amsmath,multicol}
 2

\newtheorem{thm}{Theorem}[section]

\newcommand{\thmref}[1]{Theorem~\ref{#1}}

\theoremstyle{remark}
\newtheorem{rmk}{Remark}[section]

\begin{document}

\title[Evaluation of convolution sums]
{Evaluation of the convolution sums $\sum_{l+15m=n} \sigma(l) \sigma(m)$ and $\sum_{3l+5m=n} \sigma(l) \sigma(m)$ 
and an application} 
\author{B. Ramakrishnan and Brundaban Sahu}
\address[B. Ramakrishnan]{Harish-Chandra Research Institute, 
       Chhatnag Road, Jhunsi,
     Allahabad -     211 019,
   India.}
\address[Brundaban Sahu]
{School of Mathematical Sciences, National Institute of Science 
Education and Research,
PO: Sainik School,
Bhubaneswar, Odisha - 751 005,
India.}

\email[B. Ramakrishnan]{ramki@hri.res.in}
\email[Brundaban Sahu]{brundaban.sahu@niser.ac.in}


\dedicatory{Dedicated to Srinivasa Ramanujan on the occasion of his 125th birth anniversary}

\begin{abstract}
We evaluate the convolution sums $\sum_{l,m\in {\mathbb N}, {l+15m=n}} \sigma(l) \sigma(m)$ and 
$\sum_{l,m\in {\mathbb N}, {3l+5m=n}} \sigma(l) \sigma(m)$ for all $n\in {\mathbb N}$ using the theory of quasimodular forms and use 
these convolution sums to determine the number of representations of a positive integer $n$ by the form 
$$
x_1^2 + x_1x_2 + x_2^2 + x_3^2 + x_3x_4 + x_4^2 + 5 ( x_5^2 + x_5x_6 + x_6^2 + x_7^2 + x_7x_8 + x_8^2).
$$
We also determine the number of representations of positive integers by the quadratic form
$$
x_1^2 + x_2^2+x_3^2+x_4^2 + 6 (x_5^2+x_6^2+x_7^2+x_8^2),
$$
by using the convolution sums obtained earlier by Alaca, Alaca and Williams \cite{{aw3}, {aw4}}.
\end{abstract}


\maketitle

\section{Introduction}
Following \cite{{w1}, {royer}}, for $n,N\in {\mathbb N}$, we 
define $W_N(n)$ as follows.
\begin{equation}\label{WN}
W_N(n) = \sum_{m<n/N}\sigma(m) \sigma(n-Nm),
\end{equation}
where $\sigma_r(n)$ is the sum of the $r$-th powers of the divisors of $n$. We write $\sigma_1(n) =\sigma(n)$. Also, following \cite{aw1}, 
we define $W_{a,b}(n)$ for $a,b\in {\mathbb N}$ by 
\begin{equation}
W_{a,b}(n) := \displaystyle{\sum_{l,m \atop{al+bm=n}}} \sigma(l)\sigma(m).
\end{equation}
Note that $W_{1,N}(n) = W_{N,1}(n) = W_N(n)$. 
These type of sums were evaluated as early as the 19th century. For example, the sum $W_1(n)$ was evaluated by Besge, Glaisher and 
Ramanujan \cite{{besge}, {glaisher}, {ramanujan}}. 

The convolution sums $W_N(n)$ (for $1\le N\le 24$  with a few exceptions) and $W_{a,b}(n)$ for $(a,b)\in\{(2,3),$ $(3,4), (3,8), (2,9)\}$ 
have been evaluated  by using either elementary methods or analytic methods (which use ideas of Ramanujan)  or algebraic methods (using 
quasimodular forms)  (cf. \cite{{besge}, {glaisher}, {ramanujan}, {macmahon}, {lahiri1}, {lahiri2}, {melfi}, {huard}, {aw1}, {aw2}, {aw3}, {aw4}, {aw5}, {aw6}, {cw1}, {cw2}, {lw}, {w1}, {w2}, {royer}}).  Evaluation of these convolution sums has been applied to find the number of representations of integers by certain quadratic forms (cf. \cite{{huard}, {aw1}, {aw2}, {aw3}, {aw4}, {aw5}, {w1}, {w2}}). In \cite{royer}, Royer used the theory 
of quasimodular forms, especially the structure of the space of quasimodular forms (see Eq. \eqref{structure} below), to evaluate the convolution sums $W_N(n)$ for $1\le N \le 14$, except for $N=12$. For a list of evaluation of the convolution sums $W_N(n)$, we refer the reader to Table 1 in \cite{royer}. 
In this article, following the method of Royer, we evaluate the convolution sums $W_{15}(n)$ and 
$W_{3,5}(n)$ by using the theory of quasimodular forms. The evaluation of these convolution sums is then used to determine the number 
of representations of a positive integer by a certain quadratic form. More precisely, we use these convolution sums to determine the number of representations of integers by the quadratic form 
$x_1^2 + x_1x_2 + x_2^2 + x_3^2 + x_3x_4 + x_4^2 + 5 ( x_5^2 + x_5x_6 + x_6^2 + x_7^2 + x_7x_8 + x_8^2)$.  We also give a formula 
for the number of representations of integers by the quadratic forms $x_1^2 + x_2^2+x_3^2+x_4^2 + k(x_5^2+x_6^2+x_7^2+x_8^2)$,  
$k=3,6$, by using the convolution sums $W_{3}(n)$, $W_6(n)$, $W_{12}(n)$, $W_{24}(n)$, $W_{2,3}(n)$ and $W_{3,4}(n)$ evaluated by  
K. S. Williams and his co-authors \cite{{aw1}, {aw3}, {aw4}, {huard}}. The formula for $k=3$ was obtained by Alaca-Williams \cite{aw7}, where 
the terms corresponding to the cusp forms are different from our formula. The referee has informed us that the evaluation when $k=6$  
has been carried out in a similar manner by K\"okl\"uce. 


\section{Evaluation of $W_{a,b}(n)$ and some applications}

\subsection{Evaluation of $W_{15}(n)$ and $W_{3,5}(n)$}  
In this section, following Royer \cite{royer}, we evaluate the convolution sums $W_{15}(n)$ and $W_{3,5}(n)$ by using the theory of quasimodular forms. As an application, we use these convolution sums together with the convolution sum $W_5(n)$ derived by Lemire and Williams \cite{lw} to obtain a formula for the number of representations of a positive integer $n$ by the quadratic form $Q$ given by:
\begin{equation}\label{Q}
Q: x_1^2 + x_1x_2 + x_2^2 + x_3^2 + x_3x_4 + x_4^2 + 5 ( x_5^2 + x_5x_6 + x_6^2 + x_7^2 + x_7x_8 + x_8^2).
\end{equation}
Let 
\begin{equation}\label{new5}
\begin{split}
\Delta_{4, 5}(z) &= [\Delta(z)\Delta(5z)]^{1/6} ~=~ \eta^4(z) \eta^4(5z) ~=~ q \prod_{n=1}^\infty (1-q^n)^4 (1-q^{5n})^4\\ 
&= \sum_{n\ge 1} \tau_{4,5}(n) q^n 
\end{split}
\end{equation}
be the normalized newform of weight $4$ on $\Gamma_0(5)$ (see \cite{royer}), where $q= e^{2\pi iz}$. In the above, $\eta(z) = q^{1/24}\prod_{n=1}^\infty (1-q^n)$ is the Dedekind eta-function and the Ramanujan function $\Delta(z) = \eta^{24}(z)$ is the normalized 
cusp form of weight $12$ on the full modular group $SL_2({\mathbb Z})$. The following theorem was proved by Lemire and Williams \cite{lw}:
\begin{thm}\label{thm:lw}
\begin{equation}
W_5(n)  = \frac{5}{312}\sigma_3(n)+\frac{125}{132}\sigma_3\left(\frac{n}{5}\right)+\frac{5-6n}{120} \sigma(n) +\frac{1-6n}{24} 
\sigma\left(\frac{n}{5}\right)-\frac{1}{130}\tau_{4, 5}(n).
\end{equation}
\end{thm}

In order to evaluate $W_{15}(n)$ and $W_{3,5}(n)$, we use the structure theorem on quasimodular forms of weight $k$ and depth $\le k/2$. 
Let $k\ge 2$ and $N\ge 1$ be natural numbers. Let $M_k(\Gamma_0(N))$ denote the ${\mathbb C}$-vector space of modular forms of weight $k$ on the congruence subgroup $\Gamma_0(N)$. For details on modular forms of integral weight we refer the reader to \cite{serre, shimura, 1-2-3}. We now define quasimodular forms. A  complex valued holomorphic function $f$ defined on the upper half-plane ${\mathcal H}$ is called a quasimodular form of weight $k$, depth $s$ ($s$ is a non-negative integer), if  there exist holomorphic functions $f_0, f_1, \ldots 
f_s$ on ${\mathcal H}$ such that 
$$
(cz+d)^{-k} f\left(\frac{az+b}{cz+d}\right) = \sum_{i=0}^s f_i(z) \left(\frac{c}{cz+d}\right)^i,
$$
for all $\begin{pmatrix} a&b \\ c& d\\ \end{pmatrix} \in \Gamma_0(N)$ and such that $f_s$ is holomorphic at the cusps and not 
identically vanishing. It is a fact that the depth of a quasimodular form of weight $k$ is less than or equal to $k/2$. For details on 
quasimodular forms we refer to \cite{kz, mr, 1-2-3}. The Eisenstein series $E_2$, which is a  
quasimodular form of weight $2$, depth $1$ on $SL_2({\mathbb Z})$ is given by 
$$
E_2(z) = 1 -24\sum_{n\ge 1} \sigma(n) e^{2\pi inz}
$$
and this fundamental quasimodular form will be used in our results. The space of quasimodular forms of weight $k$, depth $\le k/2$ on $\Gamma_0(N)$ is denoted by $\tilde{M}_k^{\le k/2}(\Gamma_0(N))$. We need the following structure theorem (see \cite{{kz},{mr}}). For 
an even integer $k$ with $k\ge 2$, we have 
\begin{equation}\label{structure}
\tilde{M}_k^{\le k/2}(\Gamma_0(N)) = \bigoplus_{j=0}^{k/2 -1} D^jM_{k-2j}(\Gamma_0(N)) \oplus {\mathbb C} D^{k/2 -1}E_2,
\end{equation}
where the differential operator $D$ is defined by $D := \frac{1}{2\pi i} \frac{d}{dz}$. Using this one can express each quasimodular form 
of weight $k$ and depth $\le k/2$ as a linear combination of $j$-th derivatives of modular forms of weight $k-2j$ on $\Gamma_0(N)$, $0\le j\le k/2-1$ and the $(k/2 -1)$-th derivate of the quasimodular form $E_2$. 

We need the following newforms of weights $2$ and $4$ on $\Gamma_0(15)$ in order to use the structure of the space $\tilde{M}_4^{\le 2}(\Gamma_0(15))$ to prove our theorem. These newforms are either eta-products or eta-quotients or linear combinations of these.  
Define the functions $\Delta_{4,15;j}(z)$, $j=1,2$ and $\Delta_{2,15}(z)$ as follows.
\begin{equation}\label{new15}
\begin{split}
\Delta_{4,15;1}(z) & :=  \Delta_{4,5}(z) + 9 \Delta_{4,5}(3z) + 5 \Delta_{2,15}(z)^2 + 2\frac{\eta(z)^5\eta(15z)^5}{\eta(3z)\eta(5z)} \\
&=  \sum_{n\ge 1} \tau_{4,15;1}(n) q^n,\\
\Delta_{4,15;2}(z) & :=  \Delta_{4,5}(z) + 9 \Delta_{4,5}(3z) + 7 \Delta_{2,15}(z)^2 = \sum_{n\ge 1} \tau_{4,15;2}(n) q^n,  
\end{split}
\end{equation}
\begin{equation}\label{new2:15}
\Delta_{2,15}(z) = \eta(z)\eta(3z)\eta(5z)\eta(15z),
\end{equation} 
where $\Delta_{4,5}(z)$ is the newform given by \eqref{new5}. 
In \cite{dummit, gordon}, conditions are given in order to determine the modularity of an eta-quotient (with weight, level and character). 
Using these conditions, it follows that the functions $\Delta_{4,15;j}(z)$, $j=1,2$ are cusp forms of weight $4$ on $\Gamma_0(15)$ and 
$\Delta_{2,15}(z)$ is a cusp form of weight $2$ on $\Gamma_0(15)$. We now show that these cusp forms are newforms in the respective 
spaces of cusp forms. A theorem of J. Sturm \cite{sturm} states that the Fourier coefficients upto $\frac{k}{12} \times i_N$ determines a modular form of weight $k$ on $\Gamma_0(N)$,  where $i_N$ denotes the index of $\Gamma_0(N)$ in $SL_2({\mathbb Z})$.  The first few Fourier coefficients of newforms of given weight and level are obtained using the 
database of $L$-functions, modular forms, and related objects (see \cite{lmfdb}). Comparing the Fourier coefficients obtained from the 
database with the Fourier coefficients of the cusp forms defined in \eqref{new15}, \eqref{new2:15}, we conclude that the forms 
$\Delta_{4,15;j}(z)$, $j=1,2$ and $\Delta_{2,15}(z)$ are newforms. 

The following are the main theorems of this section. 
\begin{thm}\label{w15-35}
Let $n \in \mathbb N,$ then
\begin{eqnarray*}
W_{15}(n) &=& \frac{1}{624}\sigma_3(n) +\frac{3}{208}\sigma_3\left(\frac{n}{3}\right) +\frac{25}{624}\sigma_3\left(\frac{n}{5}\right) 
+\frac{75}{208}\sigma_3\left(\frac{n}{15}\right) \\
&& ~~ + \frac{5-2n}{120}\sigma(n)+\frac{1-6n}{24}\sigma\left(\frac{n}{15}\right)-\frac{1}{455}\tau_{4,5}(n)-\frac{9}{455}\tau_{4,5}\left(\frac{n}{3}\right)\\
&& \quad -\frac{1}{84}\tau_{4, 15; 1}(n)-\frac{1}{80}\tau_{4, 15; 2}(n),\\
W_{3, 5}(n) &=& \frac{1}{624}\sigma_3(n)+\frac{3}{208}\sigma_3\left(\frac{n}{3}\right) +\frac{25}{624}\sigma_3\left(\frac{n}{5}\right) +\frac{75}{208}\sigma_3\left(\frac{n}{15}\right) \\
&& ~~+ \frac{5-6n}{120}\sigma\left(\frac{n}{3}\right)+\frac{1-2n}{24}\sigma\left(\frac{n}{5}\right) - \frac{1}{455}\tau_{4,5}(n)
-\frac{9}{455}\tau_{4,5}\left(\frac{n}{3}\right) \\
&& \qquad -\frac{1}{84} \tau_{4, 15; 1}(n)+\frac{1}{80}\tau_{4, 15; 2}(n).
\end{eqnarray*}
\end{thm}
\begin{proof}
Let $E_k$ denote the normalized Eisenstein series of weight $k$ on $SL_2({\mathbb Z})$ (see \cite{serre} for details). When $k=4$, the 
Eisenstein series $E_4$ has the following Fourier expansion. 
\begin{equation}
E_4(z) = 1 + 240 \sum_{n\ge 1} \sigma_3(n) q^n. 
\end{equation}
Using the structure of $\tilde{M}_4^{\le 2}(15)$ from \eqref{structure}, we get 
\begin{equation}\label{structure1} 
\tilde{M}_4^{\le 2}(\Gamma_0(15))= M_4(\Gamma_0(15))\oplus DM_2(\Gamma_0(15)) \oplus {\mathbb C}DE_2.
\end{equation}
Using the dimension formula (see for example \cite{martin, shimura}), it follows that the 
vector space $M_4(\Gamma_0(15))$ has dimension $8$ (a basis of this vector space contains $4$ non-cusp forms and 
$4$ cusp forms) and the vector space $M_2(\Gamma_0(15))$ has dimension $4$ (a basis of  this space contains $3$ non-cusp forms and 1 
cusp form). Now, it is easy to see that the set 
$$
\{E_4(z), E_4(3z), E_4(5z), E_4(15z),  \Delta_{4,5}(z), \Delta_{4,5}(3z), \Delta_{4,15;1}(z), \Delta_{4,15;2}(z)\}
$$
forms a basis of the space $M_4(\Gamma_0(15))$ and the set 
$$
\{\Phi_{1, 15}(z), \Phi_{5, 15}(z), \Phi_{1, 3}(z), \Delta_{2, 15}(z)\}
$$ 
forms a basis of the space $M_2(\Gamma_0(15))$, where
\begin{equation}
\Phi_{a, b}(z):= \frac{1}{b-a}(bE_2(bz)-aE_2(az)).
\end{equation}
Consider the quasimodular form $E_2(z) E_2(15z)$ which belongs to $\tilde{M}^{\le 2}_4(\Gamma_0(15))$. Therefore, using \eqref{structure1} 
and the bases mentioned above, we have
\begin{eqnarray*}
E_2(z)E_2(15z) & = & \frac{1}{260}E_4(z)+ \frac{9}{260}E_4(3z)+ \frac{5}{52}E_4(5z)+ \frac{45}{52}E_4(15z)  \\
&& ~~ - \frac{576}{455}\Delta_{4,5}(z)- \frac{5184}{455}\Delta_{4,5}(3z) -\frac{48}{7}  \Delta_{4, 15; 1}(z) \\
&& ~~ -\frac{36}{5}\Delta_{4, 15; 2}(z) + \frac{28}{5} D\Phi_{1, 15}(z)+\frac{4}{5}DE_2(z).
\end{eqnarray*}
Similarly, considering $E_2(3z)E_2(5z)$, which is a quasimodular form of weight $4$, depth $2$ and level $15$, we get 
\begin{eqnarray*}
E_2(3z)E_2(5z) &=& \frac{1}{260}E_4(z)+ \frac{9}{260}E_4(3z)+ \frac{5}{52}E_4(5z)+ \frac{45}{52}E_4(15z)\\
&& ~~ - \frac{576}{455}\Delta_{4,5}(z)- \frac{5184}{455}\Delta_{4,5}(3z) -\frac{48}{7} \Delta_{4, 15; 1}(z)  +\frac{36}{5}\Delta_{4, 15; 2}(z) \\
&& ~~~ + \frac{28}{5}D\Phi_{1, 15}(z) - 4D\Phi_{5, 15}(z)+\frac{4}{5}D\Phi_{1, 3}(z) +\frac{4}{5}DE_2(z).
\end{eqnarray*}
By comparing the $n$-th Fourier coefficients, we get the required the convolution sums.
\end{proof}

\subsection{Application to the number of representations}
In this section we apply the convolution sums $W_{15}(n)$ and $W_{3,5}(n)$ to derive the following theorem. Our method of 
proof is similar to that used by Alaca-Alaca-Williams (see for example \cite{aw3, aw1, aw2}). 

\begin{thm}\label{N1}
The number of representations of a positive integer $n$ by the quadratic form $x_1^2+x_1x_2+x_2^2+x_3^2+x_3x_4+x_4^2+5(x_5^2+x_5x_6+x_6^2+x_7^2+x_7x_8+x_8^2)$ is 
equal to 
\begin{eqnarray*}
\frac{12}{13}\sigma_3(n)+\frac{108}{13}\sigma_3\left(\frac{n}{3}\right) + \frac{300}{13} \sigma_3\left(\frac{n}{5}\right) + \frac{2700}{13} \sigma_3\left(\frac{n}{15}\right) + \frac{72}{91} \tau_{4,5}(n) & \\
+ \frac{648}{91} \tau_{4,5}\left(\frac{n}{3}\right) + \frac{72}{7} \tau_{4,15;1}(n).&\\
\end{eqnarray*}
\end{thm}
\begin{proof}
Let $ {\mathbb N}_0={\mathbb N} \cup \{0\}.$ For $l \in \mathbb N_0, $ let 
$$
r(l)=\# \left\{(x_1, x_2, x_3, x_4) \in \mathbb Z^4 | x_1^2+x_1x_2+x_2^2+x_3^2+x_3x_4+x_4^2=l \right\}
$$
so that $r(0)=1$.   For $l \in \mathbb N, $ we know that (see  \cite{huard})
$$
r(l)=12 \sum_{d| l, \atop 3 \not| d} d= 12 \sigma(l)-36 \sigma\left(\frac{l}{3}\right).
$$
Let $N(n)$ be the number of representations of the given quadratic form $Q$ defined by \eqref{Q}. 
Then $N(n)$ is given by 
\begin{eqnarray*}
N(n) &=&\sum_{l. m \in \mathbb N_0 \atop l +5m=n} \big( \sum_{(x_1, x_2, x_3, x_4) \in \mathbb Z^4 \atop{x_1^2+x_1x_2+x_2^2+x_3^2+x_3x_4+x_4^2=l }}  1\big)
 \big( \sum_{(x_5, x_6, x_7, x_8) \in \mathbb Z^4 \atop{x_5^2+x_5x_6+x_6^2+x_7^2+x_7x_8+x_8^2=m}} 1\big) \\
 &=& r(0)r\left(\frac{n}{5}\right) + r(n)r(0)+ \sum_{l, m \in \mathbb N \atop{l+5m=n}} r(l)r(m)\\
&=&12\sigma\left(\frac{n}{5}\right)-36\sigma\left(\frac{n}{15}\right)+12\sigma(n)-36\sigma\left(\frac{n}{3}\right)\\
 && \quad + \sum_{l, m \in \mathbb N \atop{l+5m=n}} \left(12 \sigma(l)-36 \sigma\left(\frac{l}{3}\right) \right) \left(    12 \sigma(m)-36 \sigma\left(\frac{m}{3}\right)   \right) \\
 &=& 12\sigma\left(\frac{n}{5}\right) - 36\sigma\left(\frac{n}{15}\right)+12\sigma(n) - 36\sigma\left(\frac{n}{3}\right) 
+ 144\sum_{l, m \in \mathbb N \atop{l+5m=n}} \sigma(l) \sigma(m) \\
&& ~~ - 432  \sum_{l, m \in \mathbb N \atop{l+5m=n}} \!\!\sigma(l) \sigma\left(\frac{m}{3}\right) 
- 432  \sum_{l, m \in \mathbb N \atop{l+5m=n}}  \!\!\sigma\left(\frac{l}{3}\right) \sigma(m)\\
&& \qquad +1296  \sum_{l, m \in \mathbb N \atop{l+5m=n}} 
\!\!\sigma\left(\frac{l}{3}\right)  \sigma\left(\frac{m}{3}\right)\\
 &=&12\sigma\left(\frac{n}{5}\right)-36\sigma\left(\frac{n}{15}\right)+12\sigma(n)-36\sigma\left(\frac{n}{3}\right)\\
  &&\quad + 144~ W_5(n)-432~ W_{15}(n) - 432~ W_{3, 5}(n)+1296~ W_5\left(\frac{n}{3}\right).
\end{eqnarray*}
Substituting the convolution sums using \thmref{thm:lw} and \thmref{w15-35}, we get the required formula for $N(n).$
\end{proof}

\subsection{More applications}
Let  $Q_k$ be the quadratic form  $x_1^2 + x_2^2+x_3^2+x_4^2 + k (x_5^2+x_6^2+x_7^2+x_8^2)$ and $N_k(n)$ be the number of 
representations of integers $n\ge 1$ by $Q_k$. In this section we use the convolution sums derived in \cite{{aw3}, {aw4}} to derive a 
formula for $N_6(n)$. We note that for $k=2,3,4$ similar formulas were obtained earlier by Williams \cite{w2},  Alaca-Williams \cite{aw7} and 
Alaca-Alaca-Williams \cite{aw5} respectively. As mentioned in the introduction, we learnt from the referee that the evaluation of $N_6(n)$ has also been derived recently by K\"okl\"uce. 
To find $N_6(n)$ using our method, we need the convolution sums $W_6(n)$, $W_{2,3}$ and 
$W_{24}(n)$ which were derived by Alaca-Alaca-Williams and they are given in the following theorem. 

\begin{thm}\label{williams} {\rm( cf. \cite{{aw3}, {aw4}})}
\begin{eqnarray*}
W_6(n) &=& \frac{1}{120}\sigma_3({n})+\frac{1}{30}\sigma_3\left(\frac{n}{2}\right)+\frac{3}{40}\sigma_3\left(\frac{n}{3}\right)+\frac{3}{10}\sigma_3\left(\frac{n}{6}\right)\\
 && \quad + \frac{1-n}{24}\sigma(n)+\frac{1-6n}{24}\sigma\left(\frac{n}{6}\right)-\frac{1}{120} c_{6}(n),\\
W_{2, 3}(n) &=&\frac{1}{120}\sigma_3({n})+\frac{1}{30}\sigma_3\left(\frac{n}{2}\right)+\frac{3}{40}\sigma_3\left(\frac{n}{3}\right)+\frac{3}{10}\sigma_3\left(\frac{n}{6}\right)\\
 && \quad + \frac{1-2n}{24}\sigma\left(\frac{n}{2}\right)+\frac{1-3n}{24}\sigma\left(\frac{n}{3}\right)-\frac{1}{120} c_{6}(n),\\
W_{24}(n) &=& \frac{1}{1920}\sigma_3(n)+\frac{1}{640} \sigma_3\left(\frac{n}{2}\right) + \frac{3}{640} \sigma_3\left(\frac{n}{3}\right)
+\frac{1}{160}\sigma_3\left(\frac{n}{4}\right)\\
&& \quad + \frac{9}{640}\sigma_3\left(\frac{n}{6}\right)+ \frac{1}{30}\sigma_3\left(\frac{n}{8}\right)+\frac{9}{160}\sigma_3\left(\frac{n}{12}\right)+  \frac{3}{10}\sigma_3\left(\frac{n}{24}\right) \\
 &&\quad + \frac{4-n}{96}\sigma(n)+\frac{1-6n}{24}\sigma\left(\frac{n}{24}\right)-\frac{61}{1920} c_{1, 24}(n),
\end{eqnarray*}
where $c_6(n)$ and $c_{1,24}(n)$ are the $n$-th Fourier coefficients of weight $4$ normalized newforms which are given in \cite[p. 492]{aw3} and  \cite[p. 94]{aw4} respectively. 
\end{thm}

In the following we use \thmref{williams} to derive a formula for $N_6(n)$. 
\begin{thm}\label{n6}
The number of representations of a positive integer $n$ by the quadratic form $Q_6$ is given by 
\begin{eqnarray*}
N_6(n) &=& \frac{2}{5} \sigma_3(n) - \frac{2}{5} \sigma_3\left(\frac{n}{2}\right) + \frac{18}{5} \sigma_3\left(\frac{n}{3}\right) - 
\frac{8}{5} \sigma_3\left(\frac{n}{4}\right) - \frac{18}{5} \sigma_3\left(\frac{n}{6}\right)  \\
&& ~~ + \frac{128}{5} \sigma_3\left(\frac{n}{8}\right)  - \frac{72}{5} \sigma_3\left(\frac{n}{12}\right) 
+ \frac{1152}{5} \sigma_3\left(\frac{n}{24}\right) - \frac{8}{15} c_6(n) \\
&& ~~ + \frac{32}{15} c_6\left(\frac{n}{2}\right) - \frac{128}{15} c_6\left(\frac{n}{4}\right) + \frac{122}{15} c_{1,24}(n).  \\
\end{eqnarray*}
\end{thm}
\begin{proof}
For $l \in \mathbb N_0, $ let 
$$
r_4(l)=\# \left\{(x_1, x_2, x_3, x_4) \in \mathbb Z^4 | x_1^2+x_2^2+x_3^2+x_4^2=l \right\}
$$
so that $r(0)=1$.   For $l \in \mathbb N, $ we know the formula due to Jacobi  (see  \cite{grosswald})
$$
r_4(l)=8 \sum_{d| l, \atop 4 \not| d} d= 8 \sigma(l)-32 \sigma\left(\frac{l}{4}\right).
$$
Then $N_6$ is given by 
\begin{eqnarray*}
N_6(n) &=&\sum_{l. m \in \mathbb N_0 \atop l +6m=n} \big( \sum_{(x_1, x_2, x_3, x_4) \in \mathbb Z^4 
\atop{x_1^2+x_2^2+x_3^2+x_4^2=l }}  1\big)
 \big( \sum_{(x_5, x_6, x_7, x_8) \in \mathbb Z^4 \atop{x_5^2+x_6^2+x_7^2+x_8^2=m}} 1\big) \\
 &=& r_4(0)r_4\left(\frac{n}{6}\right) + r_4(n)r_4(0)+ \sum_{l, m \in \mathbb N \atop{l+6m=n}} r_4(l)r_4(m)\\
 &=&8\sigma\left(n\right) -32 \sigma\left(\frac{n}{4}\right)+8\sigma\left(\frac{n}{6}\right) - 32\sigma\left(\frac{n}{24}\right)\\
 &&\quad + \sum_{l, m \in \mathbb N \atop{l+6m=n}} \left(8 \sigma(l)-32 \sigma\left(\frac{l}{4}\right) \right) \left( 8 \sigma(m)-32 \sigma\left(\frac{m}{4}\right)   \right) \\
&=&8\sigma\left(n\right) -32 \sigma\left(\frac{n}{4}\right)+8\sigma\left(\frac{n}{6}\right) - 32\sigma\left(\frac{n}{24}\right) + 64\sum_{l, m \in \mathbb N \atop{l+6m=n}} \sigma(l) \sigma(m) \\
&& \quad  - 256  \sum_{l, m \in \mathbb N \atop{l+6m=n}} \!\!\sigma(l) \sigma\left(\frac{m}{4}\right) 
- 256  \sum_{l, m \in \mathbb N \atop{l+6m=n}}  \!\!\sigma\left(\frac{l}{4}\right) \sigma(m)\\
&& \qquad +1024  \sum_{l, m \in \mathbb N \atop{l+6m=n}} 
\!\!\sigma\left(\frac{l}{4}\right)  \sigma\left(\frac{m}{4}\right)\\
 &=&8\sigma\left(n\right) -32 \sigma\left(\frac{n}{4}\right)+8\sigma\left(\frac{n}{6}\right) - 32\sigma\left(\frac{n}{24}\right) + 64 ~W_6(n) \\
&& \qquad + 1024~ W_6\left(\frac{n}{4}\right)  -256~ W_{24}(n) - 256~ W_{2, 3}\left(\frac{n}{2}\right). 
\end{eqnarray*}
Substituting the convolution sums from \thmref{williams} in the above gives the required formula for $N_6(n)$.
\end{proof}
\begin{rmk}
{\rm 
The representation numbers $N_k(n)$ for $k=2,4$ were obtained by Williams \cite{w2} and by Alaca-Alaca-Williams \cite{aw5} using the 
convolution sums $W_{1,8}(n)$, $W_{1,16}(n)$  and for $k=3$ it was derived by Alaca-Williams \cite{aw7} as a consequence of the representation of positive integers by certain octonary quadratic forms. Note that $N_3(n)$ can also be obtained  in a similar way as done in 
the cases $k=2,4$. In fact, 
\begin{eqnarray}
N_3(n) &=& 8\sigma(n)-32\sigma\left(\frac{n}{4}\right)+8\sigma\left(\frac{n}{3}\right)  -32\sigma\left(\frac{n}{12}\right) \\
&& \quad + 64~ W_3(n) + 1024~ W_3\left(\frac{n}{4}\right) - 256~ W_{12}(n) - 256~ W_{3, 4}(n).\nonumber 
\end{eqnarray}
Using the convolution sums $W_3(n)$, $W_{3,4}(n)$ and $W_{12}(n)$ obtained in \cite{{huard}, {aw1}}, we have the following formula for 
$N_3(n)$:
\begin{eqnarray}\label{n3}
N_3(n) & = & \frac{8}{5} ~\sigma_3(n) - \frac{16}{5} ~\sigma_3\left(\frac{n}{2}\right)  + \frac{72}{5} ~\sigma_3\left(\frac{n}{3}\right) 
+ \frac{128}{5} ~\sigma_3\left(\frac{n}{4}\right) \\
&& ~~ - \frac{144}{5}~ \sigma_3\left(\frac{n}{6}\right)  + \frac{1152}{5}~ \sigma_3\left(\frac{n}{12}\right) + \frac{88}{15}~ c_{1,12}(n) 
+ \frac{8}{15}~ c_{3,4}(n). \nonumber 
\end{eqnarray}
The difference between the formula given in \cite[Theorem 1.1 (ii)]{aw7} and \eqref{n3} is due to different cusp forms used.  In \cite{aw7} 
coefficients of the newform of weight $4$ and level $6$ appear while in the above formula Fourier coefficients of two cusp forms of weight $4$ and level $12$ appear. 
}
\end{rmk}


\noindent {\bf Acknowledgements}. ~We have used the open-source mathematics software SAGE (www.sagemath.org) to do our calculations. 
The work was done during the second author's visit to the Harish-Chandra Research Institute (HRI), Allahabad. He wishes to thank HRI 
for the warm hospitality during his stay. Finally, the authors thank the referee for his/her useful suggestions.


\begin{thebibliography}{10}
\bibitem{glaisher}
 J. W. L. Glaisher, {\em On the squares of the series in which the coefficients are the sums of the divisor
 of the exponents}, Mess. Math. {\bf 15} (1885), 1--20.
\end{thebibliography}
\end{document}